\documentclass{amsart}
\usepackage{amsmath}

\usepackage{latexsym, amssymb, amsthm}

\title{Anomalous Vacillatory Learning}


\author{Achilles A. Beros}

\address{Department of Mathematics\\
University of Wisconsin - Madison\\
Madison, WI 53706}

\email{aberos@math.wisc.edu}

\newtheorem{Theorem}{Theorem}[section]

\theoremstyle{definition}
\newtheorem{Definition}[Theorem]{Definition}

\begin{document}

\begin{abstract}
In 1986, Osherson, Stob and Weinstein asked whether two variants of anomalous vacillatory learning, TxtFex$^*_*$ and TxtFext$^*_*$, could be distinguished ~\cite{STL}.  In both, a machine is permitted to vacillate between a finite number of hypotheses and to make a finite number of errors.  TxtFext$^*_*$-learning requires that hypotheses output infinitely often must describe the same finite variant of the correct set, while TxtFex$^*_*$-learning permits the learner to vacillate between finitely many different finite variants of the correct set.  In this paper we show that TxtFex$^*_*$ $\neq$ TxtFext$^*_*$, thereby answering the question posed by Osherson, \textit{et al}.  We prove this in a strong way by exhibiting a family in TxtFex$^*_2 \setminus \mbox{TxtFext}^*_*$.
\end{abstract}

\maketitle

\section{Introduction}

In order to prove TxtFex$^*_*$ $\neq$ TxtFext$^*_*$, we explicitly construct a family that is TxtFex$^*_2$-learnable, but not TxtFext$^*_*$-learnable.  We diagonalize against every attempt to TxtFext$^*_*$-learn the family by including, for each machine, a subfamily that witnesses the machine's failure to TxtFext$^*_*$-learn the family.  Each subfamily is produced by means of an effective construction and the entire family is uniformly computably enumerable ($u.c.e.$).\par

All sets considered are subsets of the natural numbers and all families are collections of such subsets.  We will use $\langle x, y\rangle$ to denote a computable pairing function.  Given natural numbers $e$ and $s$, $W_{e,s}$ will denote the result of computing the set coded by $e$, up to $s$ stages, using a fixed computable numbering of $c.e.$ sets.  By $\{ D_n \}_{n \in \mathbb N}$ we mean a fixed computable enumeration of all finite sets.  Lower case Greek letters will typically refer to strings of natural numbers.  Enumerations (called texts in learning theory) will be treated either as infinite strings or as functions on the natural numbers, depending on which is most appropriate in the given setting.  Initial segments of enumerations will feature throughout this paper and will either be denoted by lowercase Greek letters, as mentioned above, or by initial segments of functions.  To switch from finite ordered lists to unordered sets, we say that content$(\sigma) = \{x\in \mathbb N : \exists n ( x = \sigma(n) )\}$; we use the same notation when switching from enumerations to the enumerated set.  We write $A =^*B$ when the symmetric difference of $A$ and $B$ is finite.   When we wish to specify a bound on the cardinality of the symmetric difference, we write $A=^c B$, meaning that the symmetric difference of $A$ and $B$ has cardinality less than or equal to~$c \in \mathbb N$.\par

Given a fixed computable enumeration of all effective learning machines, functions from $\mathbb N^{<\mathbb N}$ to $\mathbb N$, $M_e$ denotes the learner coded by $e$.  In general, learners will be denoted by $M$. 

\begin{Definition}[\cite{POV}]
Let $M$ be a learning machine and $i,j \in \mathbb N\cup\{*\}$.  The definition of TxtFex$^i_j$-learning is in four parts.
\begin{enumerate}
\item $M$ TxtFex$^i_j$-identifies an enumeration $f$ if and only if there is a finite set $S$ with card$(S)\leq j$ such that $(\forall^{\infty} n) (\forall a \in S) (M(f\upharpoonright n) \in S \wedge W_a =^i \mbox{content}(f))$.  If $j = *$, then we place no bound on card$(S)$.
\item $M$ TxtFex$^i_j$-learns a $c.e.$ set $A$ if and only if $M$ TxtFex$^i_j$-identifies every enumeration for $A$.
\item $M$ TxtFex$^i_j$-learns a family of $c.e.$ sets if and only if $M$ TxtFex$^i_j$-learns every member of the family.
\item $\mathcal F$ is TxtFex$^i_j$-learnable (denoted $\mathcal F\in$ TxtFex$^i_j$) if and only if there is a machine $M$ that TxtFex$^i_j$-learns $\mathcal F$.
\end{enumerate}
\end{Definition}

\begin{Definition}[\cite{STL}]
Let $M$ be a learning machine and $i,j \in \mathbb N\cup\{*\}$.  The definition of TxtFext$^i_j$-learning is analogous to that of TxtFex$^i_j$. 
\begin{enumerate}
\item $M$ TxtFext$^i_j$-identifies an enumeration $f$ if and only if there is a finite set $S$ with card$(S)\leq j$ such that $(\forall^{\infty} n) (\forall a, b \in S)(M(f\upharpoonright n) \in S \wedge W_a = W_b =^i \mbox{content}(f))$.  If $j = *$, then we place no bound on card$(S)$.
\item $M$ TxtFext$^i_j$-learns a $c.e.$ set $A$ if and only if $M$ TxtFext$^i_j$-identifies every enumeration for $A$.
\item $M$ TxtFext$^i_j$-learns a family of $c.e.$ sets if and only if $M$ TxtFext$^i_j$-learns every member of the family.
\item $\mathcal F$ is TxtFext$^i_j$-learnable (denoted $\mathcal F\in$ TxtFext$^i_j$) if and only if there is a machine $M$ that TxtFext$^i_j$-learns $\mathcal F$.
\end{enumerate}
\end{Definition}

Inspection of the definitions reveals that TxtFext$^i_j$ is a weaker learning criterion than TxtFex$^i_j$ in the sense that every TxtFext$^i_j$-learnable family is also TxtFex$^i_j$-learnable, i.e. TxtFext$^i_j \subseteq \mbox{TxtFex}^i_j$.  \par

The following two theorems tantalizingly hinted that TxtFext$^*_*$ might be equivalent to TxtFex$^*_*$.  As we shall see, this is not the case.

\begin{Theorem}[Fulk, Jain, Osherson]
$(\forall i,j \in \mathbb N)(\mbox{TxtFex}^i_j \subseteq \mbox{TxtFext}^{ci}_j)$, where $c$ depends only on $j$.
\end{Theorem}

\begin{Theorem}[Fulk, Jain, Osherson]\label{fulketal}
$(\forall i \in \mathbb N)(\mbox{TxtFex}^i_* \subseteq \mbox{TxtFext}^*_*)$.
\end{Theorem}

For proofs of these theorems, see ~\cite{OP}.\par

In our concluding remarks, we shall make use of Theorem \ref{fulketal} together with our own result to describe an interesting relationship between the two notions of anomalous vacillatory learning.

\section{TxtFex$^*_2$ $\neq$ TxtFext$^*_*$}

We begin with a heuristic overview of the diagonalization process.  Intuitively, we are searching for a string, $\sigma$, on which the learner commits to hypothesizing a finite number of different codes for the same set on all extensions of $\sigma$.  Such a string may not exist, but the construction will be such that if no string can be found, then the family under construction will include a set, on some enumeration of which, the machine never commits to output only hypotheses that code a single set.  On the other hand, if $\sigma$ does exist, the construction will produce two sets in the family that contain content$(\sigma)$ and have infinite symmetric difference.\par

We treat each step of the diagonalization as indexed by $e$ and consider the learner, $M_e$.  That step of the diagonalization will produce a family, $\mathcal L_e$, that $M_e$ cannot TxtFext$^*_*$-learn.

\begin{Theorem}\label{main}
There is a $u.c.e.$ family, $\mathcal L$, that is TxtFex$^*_2$-learnable, but is not TxtFext$^*_*$-learnable.
\end{Theorem}

\begin{proof}
Fix a learner, $M_e$.  We begin by describing what is needed to prevent $M_e$ from TxtFext$^*_*$-learning.  The result of this step will be a family, $\mathcal L_e$.  Let $L_e = \{e, e+1, \ldots\}$.  Depending on the course of the construction, $L_e$ may or may not be included in $\mathcal L_e$.  Motivated by our interest in strings on which $M_e$ has committed to a finite collection of hypotheses, we make the following definition.

\begin{Definition}\label{stab}
A string $\sigma$ is said to be an $(e,k)$-stabilizing sequence if and only if the following conditions are met for all $\tau \succeq \sigma$ such that content$(\tau) \subseteq L_e$ and $t\in \mathbb N$:
\begin{enumerate}
\item$\{e, e+1, \ldots, e+k\} \subseteq \mbox{content} (\sigma)$
\item$M_e(\tau) \leq |\sigma|$ 
\item$W_{M_e(\sigma),|\sigma|+t}\cap [0,k) = W_{M_e(\tau),|\sigma|+t}\cap [0,k)$.
\end{enumerate}
\end{Definition}

In essence, Definition \ref{stab} describes strings that adhere to a certain form, that define cones in $\{ \tau : \mbox{content}(\tau)\subseteq L_e \}$ on which $M_e$ outputs no new hypotheses, and on extensions of which $M_e$ outputs hypotheses for sets that are equal.  Since this last claim cannot be verified in the limit (it is a $\Pi_2^0$ predicate), the above definition describes a finite approximation.\par

The predicate ``$\sigma$ is not an $(e,k)$-stabilizing sequence" is $\Sigma^0_1$ as it requires only a witnessing string and natural number to verify.  Thus, we may define a sequence of strings that converges in the limit to an $(e,0)$-stabilizing sequence, $\sigma_{e,0}$, if such a string exists.  Extending this strategy, we will construct $\sigma_{e,n,s}$ for all $n,s \in \mathbb N$, such that 

\begin{itemize}
\item $\sigma_{e,n,s} \preceq \sigma_{e,n+1,s}$ for all $n,s\in \mathbb N$, if both strings are defined.
\item $\sigma_{e,0,0}, \sigma_{e,0,1}, \ldots$ converges to an $(e,0)$-stabilizing sequence, if one exists.
\item If $\sigma_{e,k,0}, \sigma_{e,k,1}, \ldots$ converges to a string $\sigma_{e,k}$ for all $k<n$, then \newline $\sigma_{e,n,0},\sigma_{e,n,1},\ldots$ converges to an $(e,n)$-stabilizing sequence, $\sigma_{e,n}$, that extends $\sigma_{e,k}$ for all $k<n$, if such a $\sigma_{e,n}$ exists.
\end{itemize}

Before constructing $\sigma_{e,n,s}$, we introduce some notation.  First, define the following finite set of strings.
$$A(\sigma, s) = \{\tau : (\mbox{content}(\tau)\subseteq L_e) \wedge (\max(\mbox{content}(\tau))\leq s) \wedge (|\tau| \leq s) \wedge (\tau \succeq \sigma)\}$$
Next, let $Q(e,n,\sigma,s)$ be the computable predicate ``there is no string in $A(\sigma,s)$ and natural number less than or equal to $s$ witnessing that $\sigma$ is not an $(e,n)$-stabilizing sequence".  Last, fix a symbol, ?, which will be used to indicate that a string is undefined.  We now give an effective algorithm for computing $\sigma_{e,n,s}$.\par
\medskip

\textbf{Stage 0:}  At this stage, no strings have yet been defined.  We set $\sigma_{e,0,0}$ to be the empty string.\par

\textbf{Stage s+1:}  We set $\sigma_{e,s+1,i} =~?$ for $i \leq s$.  We perform the following actions for each $n$, starting with $n = 0$, up to $n = s$.\par

\begin{enumerate}
	\item  If $\sigma_{e,n,s} \neq~?$, $\sigma_{e,i,s+1} \neq~?$ for all $i < n$, and $Q(e,n,\sigma_{e,n,s},s+1)$, then we set $\sigma_{e,n,s+1} = \sigma_{e,n,s}$.

\item  Otherwise, we consider two possibilities.

\begin{enumerate}
\item  If $\sigma_{e,i,s+1} \neq~?$ for all $i < n$ and there exists $\tau \in A(\sigma_{e,n-1,s+1},s+1)$ (where we replace $\sigma_{e,n-1,s+1}$ with the empty string if $n=0$) such that $Q(e,n,\tau,s+1)$, then we set $\sigma_{e,n,s+1}$ to be the least such $\tau$.

\item  Otherwise, we set $\sigma_{e,n,s+1} =~?$.
\end{enumerate}

\end{enumerate}

Once the process above terminates, we end the stage of the construction.\par
\medskip

Observe that $\{ \sigma_{e,n,s} \}_{s\in \mathbb N}$ converges if and only if $\{ \sigma_{e,k,s} \}_{s\in \mathbb N}$ converges for $k < n$ and there is an $(e,n)$-stabilizing sequence extending the string to which $\{ \sigma_{e,n-1,s} \}_{s\in \mathbb N}$ converges.  Furthermore, if $\{ \sigma_{e,n,s} \}_{s\in \mathbb N}$ converges, it converges to such an $(e,n)$-stabilizing sequence.\par

Define $a_{e,\ell}$ to be the least even number greater than $e+\ell+1$ such that $\sigma_{e,h,s} = \sigma_{e,h,s+1} \neq~?$ for all $h \leq \ell$ and $s \geq a_{e,\ell}$, if it exists.  Let $b_{e,\ell} = a_{e,\ell} + 1$.  These numbers will allow us to monitor the convergence of the sequences, $\{ \sigma_{e,\ell,s} \}_{s\in \mathbb N}$, and control the construction as appropriate.  If $\{ \sigma_{e,k,s} \}_{s\in \mathbb N}$ does not converge for some $k \in \mathbb N$, then $a_{e,\ell}$ will be undefined for $\ell \geq k$.\par

Define two sets
\begin{align*}
R_e &= \{x \in L_e : \forall \ell (x \neq a_{e,\ell})\} \\
\hat{R}_e &= \{x \in L_e : \forall \ell (x \neq b_{e,\ell})\}.
\end{align*}

Observe that $R_e$ is $c.e.$  Because $a_{e,0} < a_{e,1} < \ldots$ and $a_{e,\ell} \geq \ell$, we see that $x \in R_e$ if and only if $x \neq a_{e,\ell}$ for all $\ell \leq x$.  Although $a_{e,\ell}$ is not computable, $x \neq a_{e,\ell}$ is $\Sigma_1^0$.  
$$x \neq a_{e,\ell} \leftrightarrow (\sigma_{e,\ell,x} =~?) \vee (\sigma_{e,\ell,x-1} = \sigma_{e,\ell,x}) \vee (\exists s\geq x) (\sigma_{e,\ell,s} \neq \sigma_{e,\ell,x})$$
Thus, $x \in R_e$ is a finite conjunction of $\Sigma_1^0$ statements.  Similarly, substituting $b_{e,\ell}$ for $a_{e,\ell}$, we see that $\hat{R}_e$ is also $c.e.$  We are now in a position to define $\mathcal L_e$.  Recall that $D_0, D_1,\ldots$ enumerates all finite sets.
$$\mathcal L_e = \{R_e \cup (D_n \cap [e, \infty) ): n \in \mathbb N\} \cup \{\hat{R}_e \cup (D_n  \cap [e, \infty) ): n \in \mathbb N\}$$
We now return to $M_e$, the learner against which we are currently diagonalizing.  We must prove that $M_e$ is incapable of TxtFext$^*_*$-learning $\mathcal L_e$.\par

\smallskip

\textbf{Case 1:}  Suppose there is a minimal $\ell \neq 0$ such that $\sigma_{e,\ell}$ is undefined.  By definition, this implies there is no $\sigma$ extending $\sigma_{e,\ell-1}$ such that $e+\ell \in \mbox{content}(\sigma) \subset L_e$ and $W_{M_e(\sigma),|\sigma|+s}\cap [0,\ell) = W_{M_e(\tau),|\sigma|+s}\cap [0,\ell)$, for all $\tau$ such that $\sigma\prec \tau$ with $\mbox{content}(\tau) \subset L_e$.  Furthermore, since $\sigma_{e,\ell}$ is undefined, $\sigma_{e,i}$ is undefined for all $i \geq \ell$. Consequently, $a_{e,i}$ is undefined for all $i \geq \ell$ and both $R_e$ and $\hat{R}_e$ are cofinite subsets of $L_e$.  For a suitable finite set, $D_n$, we have $R_e \cup D_n = L_e$, and hence, $L_e \in \mathcal L_e$.  By repeatedly selecting extensions on which $M_e$ outputs hypotheses coding distinct sets, we can inductively build an enumeration of $L_e$ on which $M_e$ infinitely often outputs codes for at least two sets that are not equal.  If $\ell = 0$, there is the additional possibility that $M_e$ cannot be made to select a finite collection of hypotheses and restrict its output to that finite list.  The machine has failed to TxtFext$^*_*$-learn $\mathcal L_e$.\par

\smallskip

\textbf{Case 2:}  Suppose that $\sigma_{e,\ell}$ is defined for all $\ell$.  Both $R_e$ and $\hat{R}_e$ are coinfinite sets and have infinite symmetric difference.  By the definition of $\sigma_{e,0}$, for any $\tau$ such that $\sigma_{e,o}\prec \tau$ and content$(\tau) \subset L_e$, $M_e(\tau) \leq |\sigma_{e,0}|$.  We may therefore define a finite list, $h_0, h_1, \ldots, h_n$, of all distinct hypotheses that $M_e$ outputs on extensions of $\sigma_{e,0}$.  Pick $\ell$ sufficiently large so that, for each $i,j \leq n$ for which $W_{h_i} \neq W_{h_j}$, there is an $x \in W_{h_i} \triangle W_{h_j}$ such that $x < \ell$.\par

All hypotheses made by $M_e$ on extensions of $\sigma_{e,\ell}$ contained in $L_e$ must have the same intersection with $[0,\ell-1]$ as $W_{M_e(\sigma_{e,\ell})}$.  By the choice of $\ell$, all such hypotheses must code the same set, yet $\mathcal L_e$ contains two sets that extend $\sigma_{e,\ell}$ and have infinite symmetric difference: content$(\sigma_{e,\ell}) \cup R_e$ and content$(\sigma_{e,\ell}) \cup \hat{R}_e$.  Again, we witness failure by $M_e$ to TxtFext$^*_*$-learn $\mathcal L_e$.\par

\bigskip

For each $e \in \mathbb N$, we have shown that $\mathcal L_e$ is not TxtFext$^*_*$-learnable by $M_e$.  Consequently, $\mathcal L = \bigcup_{e \in \mathbb N} \mathcal L_e$ is not TxtFext$^*_*$-learnable.  We must now verify that $\mathcal L$ is indeed TxtFex$^*_2$-learnable.\par

Every set in $\mathcal L$ is a finite variant of $R_e$ or $\hat{R}_e$ for some $e \in \mathbb N$.  Therefore, a learner need only identify the appropriate $e$ and determine to which of $R_e$ and $\hat{R}_e$ the input enumeration is most similar.  Since $R_e$ is co-even and $\hat{R}_e$ is co-odd, they are identifiable by the numbers not in them. Let $x_e$ and $\hat{x}_e$ be codes for $R_e$ and $\hat{R}_e$, respectively.  For notational ease, let $m_{\sigma} = $ min(content$(\sigma)$) and $n_{\sigma} = \mbox{min}(\{y > m_{\sigma} : y \notin \mbox{content}(\sigma)\})$.  Given $\sigma$, an intial segment of an enumeration for a set in $\mathcal L$, $m_{\sigma}$ is the current guess at the least member of the set (hence the $e$ for which the set is in $\mathcal L_e$) and $n_{\sigma}$ is the current guess at the least element of $L_e$ not in the set.  Define a machine $M$ as follows:\par
$$M(\sigma) = \begin{cases}
       x_e &  \mbox{if }e = m_{\sigma} \wedge (n_{\sigma} \mbox{ is even}),\\
       \hat{x}_e &  \mbox{if }e = m_{\sigma} \wedge (n_{\sigma} \mbox{ is odd}),\\
	0 &  \mbox{otherwise}.
\end{cases}$$

 Suppose that $M$ is receiving an enumeration for $L \in \mathcal L$.  Every set in $\mathcal L$ is either of the form $R_e \cup D_n$ or $\hat{R}_e \cup D_n$, for some $e,n\in \mathbb N$.  By the symmetric relationship between $R_e$ and $\hat{R}_e$, we may assume that $L = R_e \cup D_n$ for some specific $e$ and $n$.  We must consider two cases: $R_e$ is either cofinite or coinfinite.\par

If $R_e$ is cofinite, $\hat{R}_e$ is also cofinite.  As a consequence, $R_e =^* \hat{R}_e$.  Eventually, the enumeration will exhibit the least element of the set being enumerated.  After that stage, $M$ will either output $x_e$ or $\hat{x}_e$.  Given the model of learning, both are correct hypotheses.  If $R_e$ is coinfinite, then $L_e \setminus R_e$ is an unbounded set of even numbers.  The target set is a finite variant of $R_e$.  Hence, the least element not in content$(\sigma)$ and greater than $e$ will be even for cofinitely many initial segments of any enumeration.  In other words, for cofinitely many initial segments, $\sigma$, of any enumeration of $L$, $n_{\sigma}$ is even and $M(\sigma) = x_e$, a code for a finite variant of the enumerated set.\par

We have constructed a family $\mathcal L$ such that, for each computable machine, $\mathcal L$ contains a subfamily that the machine cannot TxtFext$_*^*$-learn, and we have exhibited a specific machine that TxtFex$_2^*$-learns the whole family.  This completes the proof.\par
\end{proof}

\section{Conclusion}

Recall the statement of Theorem \ref{fulketal} from the introduction:

$$(\forall j)(\mbox{TxtFex}^j_* \subseteq \mbox{TxtFext}^*_*).$$

\medskip
\smallskip

Combining this with Theorem \ref{main}, we observe the following intriguing relationship between the anomalous versions of the two learning criteria

$$(\forall j)(\mbox{TxtFex}^j_* \subseteq \mbox{TxtFext}^*_* \subsetneq \mbox{TxtFex}^*_*).$$

\medskip
\smallskip

A great number of other results about vacillatory learning are already known.  Many of the results can be found in a paper of Case's ~\cite{POV} or in Osherson, Stob and Weinstein's book ~\cite{STL}.\par
$ $\par

\textbf{Acknowledgements.}  We would like to thank the anonymous reviewer for important comments and useful feedback on this paper.

\bibliographystyle{plain}

\end{document}